\documentclass[reqno,a4paper]{article}
\usepackage{etex}
\usepackage[T1]{fontenc}
\usepackage[utf8]{inputenc}
\usepackage{lmodern}

\usepackage{amssymb,amsmath,amsthm}
\usepackage{thmtools,xcolor}
\usepackage{units}
\usepackage{graphicx}
\usepackage[all]{xy}
\usepackage{tikz}
\usetikzlibrary{arrows,calc,positioning,decorations.pathreplacing}
\usepackage{relsize}
\usepackage{mhsetup}
\usepackage{mathtools}
\usepackage{stmaryrd}
\usepackage{tocloft}
\usepackage{paralist} 
\usepackage[left=3cm,right=3cm,top=3cm,bottom=3cm]{geometry} 
\usepackage[bottom]{footmisc}
\usepackage[colorlinks,citecolor=blue,linkcolor=blue,urlcolor=blue,filecolor=blue,breaklinks]{hyperref}
\usepackage{footnotebackref}
\usepackage[format=plain,indention=1em,labelsep=quad,labelfont={up},textfont={footnotesize},margin=2em]{caption}
\usepackage[mathscr]{euscript}
\usepackage{accents}

\definecolor{lightblue}{rgb}{0.8,0.8,1}

\numberwithin{equation}{section}
\numberwithin{figure}{section}

\definecolor{vdarkred}{rgb}{0.7,0,0}
\declaretheoremstyle[
  spaceabove=\topsep,
  spacebelow=\topsep,
  headpunct=,
  numbered=no,
  postheadspace=1ex,
  headfont=\color{vdarkred}\normalfont\bfseries,
  bodyfont=\normalfont\itshape,
]{colored}
\declaretheoremstyle[
  spaceabove=\topsep,
  spacebelow=\topsep,
  headpunct=,
  numbered=no,
  postheadspace=1ex,
  headfont=\normalfont\bfseries,
  bodyfont=\normalfont\itshape,
]{italic}
\declaretheoremstyle[
  spaceabove=\topsep,
  spacebelow=\topsep,
  headpunct=,
  numbered=no,
  postheadspace=1ex,
  headfont=\normalfont\bfseries,
  bodyfont=\normalfont\upshape,
]{upright}

\declaretheorem[style=italic,name=Theorem,numbered=yes,numberwithin=section]{thm}
\declaretheorem[style=italic,name=Lemma,numbered=yes,numberlike=thm]{lem}
\declaretheorem[style=italic,name=Proposition,numbered=yes,numberlike=thm]{prop}

\declaretheorem[style=upright,name=Remark,numbered=yes,numberlike=thm]{rmk}

\declaretheorem[style=upright,name=Notation,numbered=yes,numberlike=thm]{notation}

\newcommand{\Aut}{\mathrm{Aut}}

\newcommand{\incl}[3][right]%
{%
\draw[<-,>=#1 hook] #2 to ($ #2!0.5!#3 $);
\draw[->,>=stealth'] ($ #2!0.5!#3 $) to #3;%
}
\newcommand{\inclusion}[5][right]%
{%
\draw[<-,>=#1 hook] #4 to ($ #4!0.5!#5 $) node[#2,font=\small]{#3};
\draw[->,>=stealth'] ($ #4!0.5!#5 $) to #5;%
}

{\begin{compactitem}

}%
{\end{compactitem}}
{\begin{compactitem}[#1]

}%
{\end{compactitem}}
{\begin{compactdesc}

}%
{\end{compactdesc}}

\newcommand{\bR}{\mathbb{R}}

\renewcommand{\geq}{\geqslant}
\renewcommand{\leq}{\leqslant}
\renewcommand{\footnoterule}{%
  \kern -3pt
  \hrule width \textwidth height 0.4pt
  \kern 2.6pt
}

\definecolor{dred}{rgb}{0.7,0,0}

\definecolor{dgreen}{rgb}{0,0.5,0}

\begin{document}
\title{\LARGE\bfseries A note on representations of welded braid groups}
\date{\empty}

\author{Paolo Bellingeri\thanks{Normandie Univ., UNICAEN, CNRS, Laboratoire de Math\'ematiques Nicolas Oresme UMR CNRS~\textup{6139}, CS 14032, 14032 Cedex Cedex 5, France, \url{paolo.bellingeri@unicaen.fr}}, Arthur Souli\'e\thanks{University of Glasgow, School of Mathematics and Statistics, 132 University Pl, Glasgow G12 8TA, United-Kingdom, \url{Arthur.Soulie@glasgow.ac.uk}}}

\maketitle
{
\makeatletter
\renewcommand*{\BHFN@OldMakefntext}{}
\makeatother
}
\begin{abstract}
In this note, we adapt the procedure of the Long-Moody procedure to construct linear representations of \emph{welded} braid groups. We exhibit the natural setting in this context and compute the first examples of representations we obtain thanks to this method. We take this way also the opportunity to review the few known linear representations of welded braid groups.
\end{abstract}

\section*{Introduction}

The theory of linear representations of the braid group $\mathbf{B}_{n}$ is a very large topic. One of the most famous representations is the Burau representation \cite{burau}, which
is non faithful for $n \ge 5$.
For a long period, it was an open problem whether $\mathbf{B}_{n}$ was linear until the independent works of Bigelow \cite{bigelowbraid}, Krammer \cite{KrammerLK} and Lawrence \cite{lawrencehomological} showing a faithful representation. Since the braid group $\mathbf{B}_{n}$
is an ubiquitous object in mathematics it is natural to ask whether other generalizations are linear too,
but except some cases (for instance Artin groups of type $B$ and $D$, braid groups of the sphere and of the projective plane) this question remains widely open.

In this work note we focus on  welded braid groups, which, as braid groups, admit several different definitions, 
for instance in terms of configuration spaces of (euclidean) circles, as automorphisms of free groups, or as \emph{tubes} in $\bR^4$.

The representation theory for these groups is just at the beginning: Burau representation extends, in terms of Fox derivatives and Magnus expansion,
to welded braid groups \cite{Bardakovthestructure} and few other results are known on representations arising from braided vector spaces \cite{KMRW} and on extensions of representations of $\mathbf{B}_{n}$ to particular subgroups of 
welded braid groups \cite{BKMR19}. 

The main idea of this work is to extend Long-Moody procedure  \cite{BirmanLongMoody, Long1, Long2, Long3, soulie1, soulie2}
to $\mathbf{wB}_{n}$: in \S \ref{S1} we recall briefly the interpretation of welded braid groups in terms of fundamental groups of configuration spaces of circles
and as automorphisms of free groups through Artin homomophism; this latter interpretation will be extended to other possible representations, extending Wada representations of the classical braid group $\mathbf{B}_{n}$. In \S \ref{sec:linrep} we recall and compare different 
Burau representations for $\mathbf{wB}_{n}$, the reducible one (Proposition \ref{def:burau}), the reduced one (Proposition \ref{def:redburau}), and the dual version
 (Proposition \ref{def:dualburau}). Then we introduce Tong-Yang-Ma representations of $\mathbf{wB}_{n}$ (Proposition \ref{def:TYM}) and extending the heuristic approach proposed in \cite{TYM} we show that Burau and Tong-Yang-Ma representations are the only representations allowing a certain diagram to commute (Proposition \ref{def:unique}). The main results are given in \S \ref{sec:LM}, where we adapt to welded braid groups the Long-Moody procedure for obtaining iterated linear representations. At the first step we obtain the Burau representation (Theorem \ref{thm:recbur}) as in the case of $\mathbf{B}_{n}$. Surprisingly
 at the second iteration we do not obtain any new information, since we get the tensor product of two Burau representations (Theorem \ref{thm:iter}), while in the case of $\mathbf{B}_{n}$ we recover this way Lawrence-Krammer representation (see \cite[Section 2.3.1]{soulie1}). This result can be also compared with the fact that the "trivial" extension of Bigelow representation to $\mathbf{wB}_{n}$ (associating to "braid" generators corresponding Bigelow matrices and to
 "permutation" generators the corresponding permutations matrices)
 is not well defined (see  \cite{Bardakovthestructure, KMRW}). Finally, once more contrarily to the case of $\mathbf{B}_{n}$, we show 
 in Theorem \ref{impossiblerecoverTYM} that it is not possible to recover the Tong-Yang-Ma representation 
 for $\mathbf{wB}_{n}$ by a Long-Moody construction and we conclude with some further possible directions on the study of linear representations for $\mathbf{wB}_{n}$. In a further paper, we intend to study the Long-Moody iteration on the Tong-Yang-Ma representation, but here we focus on the two above interesting facts.

\subsubsection*{Acknowledgements}

The authors were partially supported by the ANR Project AlMaRe (ANR-19-CE40-0001-01). The first author was also partially supported by the project ARTIQ (ERDF/RIN). The second author was also partially supported by the ANR Project ChroK (ANR-16-CE40-0003). The authors are deeply grateful to Emmanuel Wagner for his suggestions and ideas all along the writing of this paper. They also thank Markus Szymik for helpful discussions on that topic.


\paragraph{Notations and conventions.}
Throughout this work, for a group defined by a (finite) presentation, we take the convention to read from the right to the left for the group operation.
 
\section{State of the art of welded braid groups representations}\label{S1}

\subsection{Recollections on welded braid groups}\label{sec:recollectionswelded}

We refer the reader to \cite{damiani2017journey} for a complete and unified presentation of the various definitions of welded braid groups, which correspond to (unextended) loop braid groups in this reference. In the following we will use essentially the interpretation of welded braid groups in terms of automorphisms of free groups and fundamental groups of configuration spaces of circles.

In this work, we focus on a 3-dimensional analogue of $\mathbf{B}_{n}$: it is the fundamental group of all configurations of $n$
unlinked Euclidean circles lying on planes that are  parallel  to a fixed one
(called \emph{untwisted rings} in \cite{BrendleHatcher}). According to \cite{BrendleHatcher} 
we will denote by $\mathcal{UR}_n$ the space of configurations of $n$
unlinked Euclidean circles being all parallel to a fixed plane 
and by $UR_n$ its fundamental group (called \emph{group of rings} in \cite{BrendleHatcher}).
The group $UR_n$ is generated by $2$ types of moves (see Figure \ref{fig:moves}).

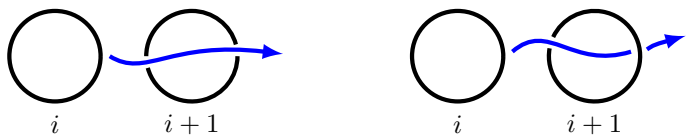
\begin{figure}[h]
\begin{center} \begin{tikzpicture}[scale=0.6]

 \draw[ultra thick]  (0:1) arc (0:360:1);

\begin{scope}[xshift=3cm]
  \draw[ultra thick]  (-165:1) arc (-165:0:1);
  \draw[ultra thick]  (-180:1) arc (-180:-345:1);
\end{scope}

\draw[->,>=latex,ultra thick,blue] (1.2,0.) .. controls (2,-0.5)  and (2.5,0.35) .. (5,0.05);

 \node at (0,-1.5) {$i$}; 
 \node at (3,-1.5) {$i+1$}; 
\end{tikzpicture} \qquad \qquad
 \begin{tikzpicture}[scale=0.6]

 \draw[ultra thick]  (0:1) arc (0:360:1);

\begin{scope}[xshift=3cm]
  \draw[ultra thick]  (-190:1) arc (-190:155:1);
\end{scope}

\draw[ultra thick,blue] (1.2,0.1) .. controls (2,0.8)  and (2.5,-0.3) .. (3.8,0.1);
\draw[->,>=latex,ultra thick,blue] (4.15,0.15) .. controls (4.3,0.25) .. (5,0.45);

 \node at (0,-1.5) {$i$}; 
 \node at (3,-1.5) {$i+1$}; 
\end{tikzpicture} \qquad
 \caption{The moves $\tau_i$ and $\sigma_i$.
 \label{fig:moves}}
 \end{center}
\end{figure}

The move $\tau_i$ is the path permuting the $i$-th and the $i+1$-th circles by passing over (or around)
while $\sigma_i$ permutes them by passing the $i$-th circle through the $i+1$-th (let us remark that our notation is different from the one of \cite{BrendleHatcher}, where $\tau_i$ was denoted by $\sigma_i$ and $\sigma_i$ by $\rho_i$; here we change the notation because 
$\sigma_i$'s generate a subgroup isomorphic to $\mathbf{B}_{n}$).

The fundamental group $UR_n$ is here denoted by $\mathbf{wB}_n$, and it is called 
welded braid group. 
Note that the convention of reading from the right to the left for the group operation is coherent with the interpretation of $\mathbf{wB}_{n}$ as the fundamental group of the space of configurations of $n$ unlinked Euclidean circles being all parallel to a fixed plane: this convention corresponds to the composition of morphisms. We abuse the notation throughout this work, identifying $\lambda\circ\lambda'=\lambda\lambda'$ for all elements $\lambda$ and $\lambda'$ of $\mathbf{wB}_{n}$.

In \cite{BrendleHatcher} was proven that the welded braid group on $n$ generators $\mathbf{wB}_{n}$ 
admits a presentation with generators $\left\{ \sigma_{i},\tau_{i}\mid i\in\left\{ 1,\ldots,n-1\right\} \right\}$
together with relations:
\begin{equation}\label{eq:relations}
\begin{cases}
\sigma_{i}\sigma_{k}=\sigma_{k}\sigma_{i} & \textrm{if \ensuremath{\mid i-k\mid\geq2},}\\
\sigma_{i}\sigma_{i+1}\sigma_{i}=\sigma_{i+1}\sigma_{i}\sigma_{i+1} & \textrm{if \ensuremath{i\in\left\{ 1,\ldots,n-2\right\} },}\\
\tau_{i}\tau_{k}=\tau_{k}\tau_{i} & \textrm{if \ensuremath{\mid i-k\mid\geq2},}\\
\tau_{i}\tau_{i+1}\tau_{i}=\tau_{i+1}\tau_{i}\tau_{i+1} & \textrm{if \ensuremath{i\in\left\{ 1,\ldots,n-2\right\} },}\\
\tau_{i}^{2}=1 & \textrm{if \ensuremath{i\in\left\{ 1,\ldots,n-1\right\} },}\\
\sigma_{i}\tau_{k}=\tau_{k}\sigma_{i} & \textrm{if \ensuremath{\mid i-k\mid\geq2},}\\
\tau_{i}\sigma_{i+1}\sigma_{i}=\sigma_{i+1}\sigma_{i}\tau_{i+1} & \textrm{if \ensuremath{i\in\left\{ 1,\ldots,n-2\right\} },}\\
\sigma_{i}\tau_{i+1}\tau_{i}=\tau_{i+1}\tau_{i}\sigma_{i+1} & \textrm{if \ensuremath{i\in\left\{ 1,\ldots,n-2\right\} }.}
\end{cases}
\end{equation}

The map $\mathcal{UR}_{n} \to \mathcal{UR}_{n+1}$ which adds a circle on the left induces an (injective) homomorphism
from $\mathbf{wB}_{n}$ to $\mathbf{wB}_{n+1}$. 

Let $\mathcal{PUR}_n$ be the space of configurations of $n$
ordered unlinked Euclidean circles being all parallel to a fixed plane; the fundamental group of 
 $\mathcal{PUR}_n$ is usually denoted by $ \mathbf{wP}_{n}$ and called welded pure braid group.

The map $\mathcal{PUR}_{n+1} \to
 \mathcal{PUR}_{n}$ which forgets the first circle is a fibration and the long exact sequence in homotopy provides the following (splitting) sequence:
 
$$\xymatrix{1\ar@{->}[r] & \mathbf{D}_{n}\ar@{->}[r] & \mathbf{wP}_{n+1}\ar@{->}[r] & \mathbf{wP}_{n}\ar@{->}[r] & 1,}
$$

where $\mathbf{D}_{n}$ consists of configurations with $n$ circles in a fixed position and the first circle varying \cite{BrendleHatcher}.

Let now $\mathcal{UR}_{1,n}$ be the orbit space 
 $\mathcal{PUR}_{n+1}/\mathfrak{S}_{n}$ where the symmetric group $\mathfrak{S}_{n}$ acts by permutation on last $n$ circles. We will denote 
 by $\mathbf{wB}_{1,n}$ its fundamental group.
 As in previous case we have a splitting sequence:

$$\xymatrix{1\ar@{->}[r] & \mathbf{D}_{n}\ar@{->}[r] & \mathbf{wB}_{1,n}\ar@{->}[r] & \mathbf{wB}_{n}\ar@{->}[r] & 1.}
$$

\subsection{Artin homomorphism}\label{Artin homomorphism}
In the following we will use another interpretation of $\mathbf{wB}_{n}$, this time in terms of automorphisms of free groups.
Let $\mathbf{F}_{n}$ be the free group on $n$ generators $\left\langle x_{1},\ldots,x_{n}\right\rangle $. We call {\em Artin homomorphism} the map $a_n\colon \mathbf{wB}_{n}\rightarrow\Aut\left(\mathbf{F}_{n}\right)$ defined as follows:
\begin{eqnarray*}
\sigma_{i} & \longmapsto & \begin{cases}
x_{i}\longmapsto x_{i+1}\\
x_{i+1}\longmapsto x_{i+1}^{-1}x_{i}x_{i+1}\\
x_{j}\longmapsto x_{j} \quad j \notin \{i, i+1\}
\end{cases}\\
\tau_{i} & \longmapsto & \begin{cases}
x_{i}\longmapsto x_{i+1}\\
x_{i+1}\longmapsto x_{i}\\
x_{j}\longmapsto x_{j} \quad j \notin \{i, i+1\}
\end{cases}
\end{eqnarray*}

This map is well defined: the relations involving only generators $\left\{ \sigma_{i}\mid i\in\left\{ 1,\ldots,n-1\right\} \right\}$
are verified since it is the usual Artin representation of $\mathbf{B}_{n}$ in $\Aut\left(\mathbf{F}_{n}\right)$. The same remark holds for 
relations involving only generators $\left\{ \tau_{i}\mid i\in\left\{ 1,\ldots,n-1\right\} \right\}$ (permutation automorphisms),
therefore the only relations that we have to verify are relations involving generators of both type $\sigma_i$ and $\tau_i$.
We check here relations $\tau_{i}\sigma_{i+1}\sigma_{i}=\sigma_{i+1}\sigma_{i}\tau_{i+1}$
and $\tau_{i+1}\tau_{i}\sigma_{i+1}=\sigma_{i}\tau_{i+1}\tau_{i}$; commutation relations are
evidently verified (see also \cite{FennRimanyiRourke}).

$$a_{n}\left(\sigma_{i+1}\sigma_{i}\tau_{i+1}\right)=\begin{cases}
x_{i}\longmapsto x_{i+2}\\
x_{i+1}\longmapsto x_{i+2}^{-1}x_{i+1}x_{i+2}\\
x_{i+2}\longmapsto x_{i+1}^{-1}x_{i}x_{i+1}\\
x_j \longmapsto x_j \quad j \notin \{i, i+1, i+2\}
\end{cases}=a_{n}\left(\tau_{i}\sigma_{i+1}\sigma_{i}\right)$$

$$a_{n}\left(\sigma_{i}\tau_{i+1}\tau_{i}\right)=\begin{cases}
x_{i}\longmapsto x_{i+2}\\
x_{i+1}\longmapsto x_{i+1}\\
x_{i+2}\longmapsto x_{i+1}^{-1}x_{i}x_{i+1}\\
x_j \longmapsto x_j \quad j \notin \{i, i+1, i+2\}
\end{cases}=a_{n}\left(\tau_{i+1}\tau_{i}\sigma_{i+1}\right)$$

Geometrically, we are associating to any 
generator of $\mathbf{wB}_{n} $ the corresponding action on the fundamental group of the ball $B^3$ less $n$ trivial circles,
which is a free group on $n$ generators. We refer to \cite{damiani2017journey} for a rigorous proof of this construction and of the fact that 
Artin homomorphism is injective. Moreover,
the image of the group $\mathbf{wB}_{n}$ in $\Aut\left(\mathbf{F}_{n}\right)$ is the subgroup of those automorphisms of $\mathbf{F}_{n}$ that send each generator of $\mathbf{F}_{n}$ to a conjugate of some generator \cite{FennRimanyiRourke}.

Notice also that the homomorphism
from $\mathbf{wB}_{n}$ to $\mathbf{wB}_{n+1}$ previously defined ("add a circle on the left")
becomes here the restriction of the map $id_{1}\ast-:\textrm{Aut}\left(\mathbf{F}_{n}\right)\hookrightarrow \textrm{Aut}\left(\mathbf{F}_{n+1}\right)$. 
 In other words, $id_{1}\ast\sigma_{i}=\sigma_{i+1}$ and $id_{1}\ast\tau_{i}=\tau_{i+1}$.

The group $\mathbf{D}_{n}$ previously defined is isomorphic to the group generated by automorphisms

$\left\{ \epsilon_{i,1} , \epsilon_{1,i} \mid i\in\left\{ 2,\ldots,n+1\right\} \right\}$ \cite{Bardakovconjugatingautomorphisms},
where

\begin{eqnarray*}
\epsilon_{1,i} & \longmapsto & \begin{cases}
x_{i}\longmapsto x_{1}^{-1} x_{i} x_{1}\\
x_{j}\longmapsto x_{j} \quad j \not= i
\end{cases}\\
\epsilon_{i,1} & \longmapsto & \begin{cases}
x_{1}\longmapsto x_{i}^{-1} x_{1} x_{i}\\
x_{j}\longmapsto x_{j} \quad j \not= 1
\end{cases}
\end{eqnarray*}

Note that $\left\{ \epsilon_{i,1} \right\}$ generate a free group of order $n$ and $\left\{ \epsilon_{1,i} \right\}$ a free abelian group of rank $n$ \cite{Bardakovconjugatingautomorphisms}, but $\mathbf{D}_{n}$ is not finitely presented when $n\ge 3$ \cite{Pe}. 
\paragraph{Alternative to Artin homomorphism.}\label{sec:Wada}
Wada \cite{wada1992group} found several \emph{local} representations of $\mathbf{B}_{n}$ in $\Aut\left(\mathbf{F}_{n}\right)$ of the following form:
any generator (and therefore its inverse) of $\mathbf{B}_{n}$ acts trivially on generators of $\mathbf{F}_{n}$ except a pair of generators:
\begin{eqnarray*} 
\sigma_i \cdot x_i=u(x_i, x_{i+1}) \, ,\\
\sigma_i \cdot x_{i+1}=v(x_i, x_{i+1}) \, , \\
\sigma_i \cdot x_j= x_j &\qquad j \not= i, i+1 \, ,
\end{eqnarray*}
where $u$ and $v$ are now words in the generators $x_i, x_{i+1}$, with $\langle x_i, x_{i+1} \rangle \simeq \mathbf{F}_{2}$. Wada found seven families of representations of local type (see Section 4 of \cite{BBrep} for a short survey on these representations), up to the dual equivalence (corresponding to the involution of $\mathbf{F}_{n}$ given by simultaneously replacing all $x_i$ with $x_i^{-1}$) and inverse equivalence (derived from the involution of $\mathbf{B}_{n}$ defined by sending $\sigma_{i}$ to $\sigma_{i}^{-1}$):\footnote{In \cite{BBrep} and \cite{wada1992group}, the action described here is the one for $\sigma_i^{-1}$. However, we choose to write the inverse symmetry equivalent here to be consistent with respect to the above Artin representation $a_n$
and with the fact that we compose elements from the right to the left.}
 
\begin{itemize}
\item Type 1, $\psi_1$: $u(x_i, x_{i+1})= x_i$ and $v(x_i, x_{i+1})= x_{i+1} $;
\item Type 2, $\psi_2$: $u(x_i, x_{i+1})= x_{i+1}^{-1}$ and $v(x_i, x_{i+1})= x_i$;
\item Type 3, $\psi_3$: $u(x_i, x_{i+1})= x_{i+1}^{-1}$ and $v(x_i, x_{i+1})= x_i^{-1} $;
\item Type 4, $\psi_{4,h}$: $u(x_i, x_{i+1})= x_{i+1}$ and $v(x_i, x_{i+1})= x_{i+1}^{-h} x_{i} x_{i+1}^{h}$;
\item Type 5, $\psi_5$: $u(x_i, x_{i+1})= x_{i+1} $ and $v(x_i, x_{i+1})= x_{i+1}x_{i}^{-1} x_{i+1}$;
\item Type 6, $\psi_6$: $u(x_i, x_{i+1})= x_{i+1}^{-1}$ and $v(x_i, x_{i+1})= x_{i+1} x_{i} x_{i+1}$;
\item Type 7, $\psi_7$: $u(x_i, x_{i+1})= x_{i}x_{i+1}^{-1}x_{i}^{-1}$ and $v(x_i, x_{i+1})= x_{i}x_{i+1}^{2}$.
\end{itemize}

We can try to extend Wada representations to welded braid groups
associating to any generator $\sigma_i$ the Wada representation of type $k$ and to any generator $\tau_i$
the corresponding permutation automorphism. We will say that a Wada representation extends to $\mathbf{wB}_{n}$
if the map defined as above on generators of $\mathbf{wB}_{n}$ is actually a homomorphism.
Type 1 does not extend, while Type 2 and Type 3 extend to $\mathbf{wB}_{n}$ but these extensions are not interesting since the image of the group generated by $\sigma_i$'s
in $\Aut\left(\mathbf{F}_{n}\right)$ is trivial or isomorphic to $\mathfrak{S}_{n}$. Between the other four representations it is easy to check (see \cite{BardakovBellingeri1}) that only Type 4 and Type 5 representations extend this way to a representation; let denote these two representations respectively by
$\chi_1$ and $\chi_2$. These two representations are not equivalent meaning that
 there are no automorphisms
$\phi \in \Aut\left(\mathbf{F}_{n}\right)$ and $\mu : \mathbf{wB}_{n} \to \mathbf{wB}_{n}$ such that
$$
\phi^{-1} \, \chi_1(\beta ) \, \phi = \chi_2(\mu(\beta )) ,
$$
for any $\beta \in \mathbf{wB}_{n}$ (see \cite{BardakovBellingeri1}).

\subsection{Known linear representations for welded braid groups}\label{sec:linrep}

We present in this section the linear representations of welded braid groups which can be straightforwardly derived from those of braid groups. In particular, we extend heuristic procedure on matrices. For the remainder of \S\ref{sec:linrep}, fix a natural number $n\geq2$.

Beforehand we indicate that we study here in the "specific" representations of welded braid groups, not those which factor through symmetric groups. Namely, recall that sending both $\sigma_{i}$ and $\tau_{i}$ in the transposition $(i,i+1)$ for $i\in\left\{ 1,\ldots,n-1\right\}$
we obtain the short exact sequence:
$$\xymatrix{1\ar@{->}[r] & \mathbf{wP}_{n}\ar@{->}[r] & \mathbf{wB}_{n}\ar@{->}[r] & \mathfrak{S}_{n}\ar@{->}[r] & 1.}$$
Therefore, all representations of the symmetric group $\mathfrak{S}_{n}$ lift to $\mathbf{wB}_{n}$, but we will not consider them since we loose too many informations on welded braid groups.

\subsubsection{The Burau representations}

The family of \emph{Burau} representations for braid groups were first introduced in \cite{burau} and has been intensively studied. We refer the reader to \cite{kasselturaev} for a complete presentation.

These representations can be extended to welded braid groups in the following way.

\begin{prop}\label{def:burau}
The following assignment defines a representation $Bur:\mathbf{wB}_{n}\rightarrow GL_{n}\left(\mathbb{Z}[t^{\pm1}]\right)$ called the Burau representation of the welded braid group $\mathbf{wB}_{n}$:
$$\sigma_{i}\mapsto Id_{i-1}\oplus\left[\begin{array}{cc}
0 & t\\
1 & 1-t
\end{array}\right]\oplus Id_{n-i-1}\,\textrm{ and }\,\tau_{i}\mapsto Id_{i-1}\oplus\left[\begin{array}{cc}
0 & 1\\
1 & 0
\end{array}\right]\oplus Id_{n-i-1}$$
for all natural numbers $i\in\left\{ 1,\ldots,n-1\right\} $.
\end{prop}

\begin{proof}
The matrices for $\sigma_{i}$ and $\tau_{i}$ define representations of braid groups and symmetric groups respectively. It follows from the relations of $\mathbf{wB}_{n}$ of \eqref{eq:relations} that we just have to check the mixed relations between the braid and symmetric generators: this is done by straightforward computations.
\end{proof}

\paragraph{Reduced version.}
As for the case of braid groups, the Burau representation is reducible and we can define an irreducible version.

\begin{prop}\label{def:redburau}
The following assignment defines a representation $\overline{Bur}:\mathbf{wB}_{n}\rightarrow GL_{n-1}\left(\mathbb{Z}[t^{\pm1}]\right)$ called the reduced Burau representation of the welded braid group $\mathbf{wB}_{n}$:
$$\sigma_{1}\mapsto\left[\begin{array}{cc}
-t & t\\
0 & 1
\end{array}\right]\oplus Id_{n-3}\,\textrm{ and }\,\tau_{1}\mapsto\left[\begin{array}{cc}
-1 & 1\\
0 & 1
\end{array}\right]\oplus Id_{n-3};$$
$$\sigma_{n-1}\mapsto Id_{n-3}\oplus\left[\begin{array}{cc}
1 & 0\\
1 & -t
\end{array}\right]\,\textrm{ and }\,\tau_{n-1}\mapsto Id_{n-3}\oplus\left[\begin{array}{cc}
1 & 0\\
1 & -1
\end{array}\right];$$
$$\sigma_{i}\mapsto Id_{i-2}\oplus\left[\begin{array}{ccc}
1 & 0 & 0\\
t & -t & 1\\
0 & 0 & 1
\end{array}\right]\oplus Id_{n-i-2}\,\textrm{ and }\,\tau_{i}\mapsto Id_{i-2}\oplus\left[\begin{array}{ccc}
1 & 0 & 0\\
1 & -1 & 1\\
0 & 0 & 1
\end{array}\right]\oplus Id_{n-i-2}$$
for all natural numbers $i\in\left\{ 2,\ldots,n-2\right\} $. Moreover we have a short exact sequence of $\mathbf{wB}_{n}$-representations
\begin{equation}\label{eq:sesburau}
\xymatrix{0\ar@{->}[r] & \overline{Bur}\ar@{->}[r] & Bur\ar@{->}[r] & \mathbb{Z}\left[t^{\pm1}\right]\ar@{->}[r] & 0}
\end{equation}
where $\ensuremath{\mathbf{wB}_{n}}$ acts trivially on the kernel $\mathbb{Z}\left[t^{\pm1}\right]$.
\end{prop}

\begin{proof}
The matrices for $\sigma_{i}$ define the reduced Burau representations of braid groups (see \cite[Section 3.3]{kasselturaev}) and those for $\tau_{i}$ define the standard representation of symmetric groups. Again the compatibility with respect to the mixed relations of $\mathbf{wB}_{n}$ are checked by straightforward computations.

Let $r_{n}$ be the $n\times n$-matrix with coefficients $r_{i,j}=1$ if $j\leq i$ and $r_{i,j}=0$ otherwise. Then
$$r_{n}\circ Bur\left(\sigma_{i}\right)\circ r_{n}^{-1}=\left[\begin{array}{cc}
\overline{Bur}\left(\sigma_{i}\right) & C_{i}\\
0 & 1
\end{array}\right]$$
where $C_{i}=\left[\begin{array}{ccccc}0 & \cdots & 0 & \delta_{i,n-1} & 1\end{array}\right]^{T}$ (where the label $^{T}$ means the transpose matrix) and $\delta_{i,n-1}$ denotes the Kronecker delta: this defines the short exact sequence \eqref{eq:sesburau}.
\end{proof}

\begin{rmk}
Since Burau representation for $\mathbf{B}_{n}$ is not faithful when $n\ge 5$, it follows that also its extension is not faithful;
indeed, the Burau representation for $\mathbf{wB}_{n}$ has non trivial kernel even for $n=2$ (see Lemma 6 of \cite{Bardakovextendingautomorphisms}).
\end{rmk}

\paragraph{Dual versions.}
Actually, there are two non-equivalent versions of the Burau representation for braid groups: the one which matrices are given in Proposition \ref{def:burau}, and its dual which matrices are the transpose of the inverse of these matrices. This dual version also lifts to the welded braid group:

\begin{prop}\label{def:dualburau}
Assigning $Bur^{*}\left(\sigma_{i}\right)=Bur^{T}\left(\sigma_{i}^{-1}\right)$ and $Bur^{*}\left(\tau_{i}\right)=Bur^{T}\left(\tau_{i}^{-1}\right)$ for all natural numbers $i\in\left\{ 1,\ldots,n-1\right\} $ defines a representation $Bur^{*}:\mathbf{wB}_{n}\rightarrow GL_{n}\left(\mathbb{Z}[t^{\pm1}]\right)$ called the dual Burau representation.

It induces a dual reduced Burau representation $\overline{Bur}^{*}:\mathbf{wB}_{n}\rightarrow GL_{n-1}\left(\mathbb{Z}[t^{\pm1}]\right)$ which is defined by $\overline{Bur}^{*}\left(\sigma_{i}\right)=\overline{Bur}^{T}\left(\sigma_{i}^{-1}\right)$ and $\overline{Bur}^{*}\left(\tau_{i}\right)=\overline{Bur}^{T}\left(\tau_{i}^{-1}\right)$ for all natural numbers $i\in\left\{ 1,\ldots,n-1\right\} $, and defines the short exact sequence of $\mathbf{wB}_{n}$-representations
\begin{equation}\label{eq:dualsesburau}
\xymatrix{0\ar@{->}[r] & \mathbb{Z}\left[t^{\pm1}\right]\ar@{->}[r] & Bur^{*}\ar@{->}[r] & \overline{Bur}^{*}\ar@{->}[r] & 0}
\end{equation}
where $\ensuremath{\mathbf{wB}_{n}}$ acts trivially on the kernel $\mathbb{Z}\left[t^{\pm1}\right]$.
\end{prop}

\begin{proof}
Taking the transpose of the inverse of a multiplication of matrices keeps the order of the multiplication. Therefore all the relations of \eqref{eq:relations} are therefore straightforwardly satisfied and proves that $Bur^{*}$ is a representation. The reduced version $\overline{Bur}^{*}$ is induced by taking the transpose of the inverse of the short exact sequence \eqref{eq:sesburau}: taking the inverse keeps the direction of the arrows, whereas the transpose reverses this direction (since it exchanges lines and columns of matrices) and we obtain \eqref{eq:dualsesburau}.
\end{proof}

The dual Burau representation was already introduced in \cite[Section 4]{Vershinin}.

\subsubsection{The Tong-Yang-Ma procedure and associated representations.}
In 1996, Tong, Yang and Ma \cite{TYM} investigated the representations of $\mathbf{B}_{n}$ where the $i$-th generator is sent to a non-singular matrix of the form
$$Id_{i-1}\oplus \left[\begin{array}{cc}
a & b\\
c & d
\end{array}\right]\oplus Id_{n-i-1}.$$In particular, they proved that there exist (up to equivalence and dual) only two non trivial representations of this type: the unreduced Burau representations and a new irreducible representation, called the \emph{Tong-Yang-Ma} representation. This last family lifts to define two families of linear representations of the welded braid group $\mathbf{wB}_{n}$.

\begin{prop}\label{def:TYM}
The following assignment defines a representation $TYM:\mathbf{wB}_{n}\rightarrow GL_{n}\left(\mathbb{Z}[t^{\pm1}]\right)$ called the Tong-Yang-Ma representation of the welded braid group $\mathbf{wB}_{n}$:
$$\sigma_{i}\mapsto Id_{i-1}\oplus\left[\begin{array}{cc}
0 & 1\\
t & 0
\end{array}\right]\oplus Id_{n-i-1}\,\textrm{ and }\,\tau_{i}\mapsto Id_{i-1}\oplus\left[\begin{array}{cc}
0 & 1\\
1 & 0
\end{array}\right]\oplus Id_{n-i-1}$$
for all natural numbers $i\in\left\{ 1,\ldots,n-1\right\} $. Taking the transpose of the inverse of the above matrices defines the dual Tong-Yang-Ma representation $TYM^{*}:\mathbf{wB}_{n}\rightarrow GL_{n}\left(\mathbb{Z}[t^{\pm1}]\right)$ (which is obviously not equivalent to $TYM$).
\end{prop}

\begin{proof}
As above for the Burau representations, we just have to check the compatibility with respect to the mixed relations \eqref{eq:relations} between the braid and symmetric generators of $\mathbf{wB}_{n}$: this is done by straightforward computations.
\end{proof}

In addition, we can carry out the analogous heuristic approach to \cite{TYM} for the welded braid groups. For all $i\in\left\{ 1,\ldots,n-1\right\} $, we denote by $\textrm{incl}_{i}^{n}:\mathbf{wB}_{2}\cong\mathbb{Z}\times\mathbb{Z}/2\mathbb{Z}\hookrightarrow\mathbf{wB}_{n}$ the inclusion morphism induced by $\textrm{incl}_{i}^{n}\left(\sigma_{1}\right)=\sigma_{i}$ and $\textrm{incl}_{i}^{n}\left(\tau_{1}\right)=\tau_{i}$.

\begin{prop}\label{def:unique}
Let $\eta_{n}:\mathbf{wB}_{n}\longrightarrow GL_{n}$ be a representation. Assume that for all $i\in\left\{ 1,\ldots,n-1\right\} $ the following diagram is commutative:$$\xymatrix{\mathbf{wB}_{n}\ar@{->}[rr]^{\eta_{n}} & & GL_{n}\left(\mathbb{Z}\left[t^{\pm1}\right]\right)\\
\mathbf{wB}_{2}\ar@{->}[rr]_{\eta_{2}}\ar@{->}[u]^{\textrm{incl}_{i}^{n}} & & GL_{2}\left(\mathbb{Z}\left[t^{\pm1}\right]\right).\ar@{->}[u]_{id_{i-1}\oplus-\oplus id_{n-i-1}}
}$$Then $\eta_{n}$ is equivalent or dual to the trivial representation, or to the (unreduced) Burau representation $Bur$, or to the Tong-Yang-Ma representation $TYM$, or else to the specialisation at $t=1$ of the Burau representation.
\end{prop}

\begin{proof}
Restricting along the natural inclusion $\mathbf{B}_{n}\hookrightarrow\mathbf{wB}_{n}$, it follows from \cite[Part II]{TYM} that three only possible matrices (up to equivalence and dual) on which the braid generators $\left\{ \sigma_{i}\right\} _{i\in\left\{ 1,\ldots,n-1\right\} }$ can be sent to are the trivial, Burau and Tong-Yang-Ma matrices.
Note that the only non-trivial representation $\mathbb{Z}/2\mathbb{Z}\rightarrow GL_{2}\left(\mathbb{Z}\left[t^{\pm1}\right]\right)$ is the permutation sending the non-trivial element of $\mathbb{Z}/2\mathbb{Z}$ to the permutation matrix.

There are thus two choices for the symmetric generators $\left\{ \tau_{i}\right\} _{i\in\left\{ 1,\ldots,n-1\right\} }$. First the relation $\tau_{i}\tau_{i+1}\tau_{i}=\tau_{i+1}\tau_{i}\tau_{i+1}$ of \eqref{eq:relations} implies that all these generators are sent to either a non-trivial matrix or the trivial matrix. Then it follows from the relation $\tau_{i}\sigma_{i+1}\sigma_{i}=\sigma_{i+1}\sigma_{i}\tau_{i+1}$ of \eqref{eq:relations} that the symmetric generators are sent to the identity matrix if the braid generators are, and to the permutation matrices otherwise.
\end{proof}


\section{The Long-Moody construction for welded braid groups}\label{sec:LM}

In $1994$, Long and Moody \cite{Long1} gave a method to construct a new linear representation of $\mathbf{B}_{n}$ from a representation of $\mathbf{B}_{n+1}$, complexifying in a sense the initial representation. For instance, it reconstructs the unreduced Burau representation from a one dimensional representation. It was studied from a functorial point of view and extended in \cite{soulie1} and then generalised to other families of groups \cite{soulie2}. In particular, the underlying framework of this method, called the \emph{Long-Moody construction}, naturally arises considering representations of \emph{welded} braid groups: the aim of this section is the study of this construction in this case. We fix a natural number $n\geq3$ all along \S\ref{sec:LM}.

\subsection{The theoretical setting of the Long-Moody construction}\label{sec:theoretical}

We detail here the required tool and present the abstract definition of the Long-Moody construction.

\paragraph{Tool.}
Recall that $\mathbf{F}_{n}=\left\langle x_{1},\ldots,x_{n}\right\rangle $ is the free group on $n$ generators.
The key ingredient to define the Long-Moody construction for welded braid groups is to find a group morphisms $\alpha_{n}\colon \mathbf{wB}_{n}\rightarrow\Aut\left(\mathbf{F}_{n}\right)$ and $\xi_{n}\colon\mathbf{F}_{n}\rightarrow\mathbf{wB}_{n+1}$ such that:
\begin{itemize}
   \item the morphism $\mathbf{F}_{n}*\mathbf{wB}_{n}\rightarrow\mathbf{wB}_{n+1}$ given by the coproduct of $\xi_{n}$ and $id_{1}*-$ factors across the canonical surjection to the semidirect product $\mathbf{F}_{n}\underset{\alpha_{n}}{\rtimes}\mathbf{B}_{n}$;
   \item the following diagram is commutative
\begin{equation}\label{eq:diagram}
\xymatrix{\mathbf{F}_{n}\ar@{^{(}->}[r]\ar@{->}[dr]_{\xi_{n}} & \mathbf{F}_{n}\underset{\alpha_{n}}{\rtimes}\mathbf{wB}_{n}\ar@{->}[d] & \mathbf{wB}_{n}\ar@{_{(}->}[l]\ar@{->}[dl]^{id_{1}*-}\\
 & \mathbf{wB}_{n+1},
}
\end{equation}
where the vertical morphism $\mathbf{F}_{n}\underset{\alpha_{n}}{\rtimes}\mathbf{wB}_{n}\rightarrow\mathbf{wB}_{n+1}$ is induced by the coproduct of $\xi_{n}$ and $id_{1}*-$ morphism $\mathbf{F}_{n}*\mathbf{B}_{n}\rightarrow\mathbf{B}_{n+1}$.
\end{itemize}

In other words, we require that for all elements $\lambda\in \mathbf{wB}_{n}$ and
$x\in \mathbf{F}_{n}$ the morphism $\xi_{n}$ satisfies the following equality in $\mathbf{wB}_{n+1}$:
\begin{equation}
\left(id_{1}*\lambda\right)\circ\xi_{n}\left(x\right)=\xi_{n}\left(\alpha_{n}\left(\lambda\right)\left(x\right)\right)\circ\left(id_{1}*\lambda\right).\label{eq:cond1}
\end{equation}

\paragraph{Definition.} The Long-Moody construction is defined as follows. We fix an abelian group $V$. We denote by $\mathcal{I}_{\mathbf{F}_{n}}$ the augmentation ideal of the group ring $\mathbb{Z}\left[\mathbf{F}_{n}\right]$. Note that the action $\alpha_{n}$ canonically induces an action of $\mathbf{wB}_{n}$ on $\mathcal{I}_{\mathbf{F}_{n}}$ (that we denote in the same way for convenience).

Let $\rho:\mathbf{wB}_{n+1}\rightarrow Aut_{\mathbb{Z}}\left(V\right)$ be a linear representation. Precomposing by the morphism $\xi_{n}$, $\rho$ gives the module $V$ a $\mathbf{F}_{n}$-module structure. Then the Long-Moody construction
$$\mathbf{LM}\left(\rho\right):\mathbf{wB}_{n}\rightarrow Aut_{\mathbb{Z}}\left(\mathcal{I}_{\mathbf{F}_{n}}\underset{\mathbf{F}_{n}}{\varotimes}V\right)$$
is the map defined by:
\[
\mathbf{LM}\left(\rho\right)\left(\lambda\right)\left(i\underset{\mathbf{F}_{n}}{\varotimes}v\right)=\left(\alpha_{n}\left(\lambda\right)\left(i\right)\underset{\mathbf{F}_{n}}{\varotimes}\rho\left(id_{1}\ast\lambda\right)\left(v\right)\right)
\]
for all $\lambda\in \mathbf{wB}_{n}$, $i\in\mathcal{I}_{\mathbf{F}_{n}}$ and $v\in V$. For sake of completeness, we detail that:

\begin{lem}\cite[Section 2.2.4]{soulie2} The representation $\mathbf{LM}\left(\rho\right)$ is well-defined.
\end{lem}

\begin{proof}
We consider elements $\lambda\in\mathbf{wB}_{n}$, $x\in\mathbf{F}_{n}$, $v\in V$ and $i\in\mathcal{I}_{\mathbf{F}_{n}}$.
First, since $\rho$ is a morphism, we deduce from \eqref{eq:cond1} that
$$\mathbf{LM}\left(\rho\right)\left(\lambda\right)\left(i\underset{\mathbf{F}_{n}}{\varotimes}\rho\left(\xi_{n}\left(x\right)\right)\left(v\right)\right)=\mathbf{LM}\left(\rho\right)\left(\lambda\right)\left(i\cdot x\underset{\mathbf{F}_{n}}{\varotimes}v\right),$$
which gives the compatibility of the assignment $\mathbf{LM}\left(\rho\right)$ with respect to the tensor product over $\mathbb{Z}\left[\mathbf{F}_{n}\right]$. Then this assignment $\mathbf{LM}\left(\rho\right)$ defines a morphism on $\mathbf{wB}_{n}$ follows from the fact that $\alpha_{n}$ and $\rho$ are themselves morphisms.
\end{proof}

The natural candidate for the morphism $\alpha_{n}$ is the Artin homomorphism $a_{n}$ recalled in \S\ref{Artin homomorphism} and we fix the assignment $\alpha_{n}=a_{n}$ from now on. We could use another Wada representation for $a_{n}$ (recalled in \S\ref{sec:Wada}), but we prove in \S\ref{impossiblerecoverTYM} that this is actually not relevant for welded braid groups.

There is always the choice of the trivial morphism $\mathbf{F}_{n}\rightarrow 0\rightarrow \mathbf{wB}_{n+1}$ as $\xi_{n}$ so that relation \eqref{eq:cond1} is satisfied. However the construction is much more interesting using a non-trivial morphism for this parameter. Indeed applying the Long-Moody construction with the trivial $\xi_{n}$ to a one-dimensional representation provides the permutation representation of $\mathbf{wB}_{n}$ (sending both the symmetric and braid generators on the permutation matrix). Moreover, the iteration of this Long-Moody construction gives the tensor powers of that permutation representation. We refer the reader to \cite[Section 2.2.5]{soulie2} for further details on that point. For that reason, the main point consists in finding non-trivial $\xi_{n}$ such that the diagram \eqref{eq:diagram} is commutative, which is the aim the following section.

\subsection{The natural example for welded braids}

We give in this section a first example of a (non-trivial) Long-Moody construction for welded braid groups. It relies on the following natural candidate for the choice of the required morphism $\xi_{n}$.

Let $\xi_{n,1}:\ensuremath{\mathbf{F}_{n}}\hookrightarrow\mathbf{wB}_{n+1}$ be the injective morphism defined by
$$x_{i}\longmapsto\left(\tau_{i-1}\cdots\tau_{2}\tau_{1}\right)^{-1}\left(\sigma_{i}\tau_{i}\right)\left(\tau_{i-1}\cdots\tau_{2}\tau_{1}\right).$$
It is natural in the sense that it identifies the free group $\mathbf{F}_{n}$ with the free subgroup of order $n$ of $\mathbf{D}_{n}$ (see \S\ref{Artin homomorphism}) generated by the elements $\left\{ \epsilon_{j,1}\right\} _{2\leq j\leq n+1}$. They are also similar to the analogous parameter for the original construction for braid groups (see \cite[Example 2.7]{soulie1}). Moreover they satisfy the key condition to define a Long-Moody construction:

\begin{lem}
The morphism $\xi_{n,1}$ satisfies the equality \eqref{eq:cond1}.
\end{lem}

\begin{proof}
Since
$$\sigma_{i+1}\sigma_{i}\tau_{i}\sigma_{i+1}^{-1}=\tau_{i}\sigma_{i+1}\sigma_{i}\tau_{i+1}\tau_{i}\sigma_{i+1}^{-1}=\tau_{i}\sigma_{i+1}\tau_{i+1}\tau_{i}$$
then
$$\xi_{n,1}\left(a_{n}\left(\sigma_{i}\right)\left(x_{i}\right)\right)=\sigma_{i+1}\xi_{n,1}\left(x_{i}\right)\sigma_{i+1}^{-1}.$$

Also, note that
$$\tau_{i}\sigma_{i+1}\tau_{i+1}\tau_{i}\sigma_{i+1}\tau_{i}\sigma_{i+1}\tau_{i+1}\tau_{i}=\sigma_{i+1}\sigma_{i}\sigma_{i+1}\tau_{i+1}\tau_{i}=\sigma_{i}\sigma_{i+1}\tau_{i+1}\tau_{i}\sigma_{i+1}$$
then
$$\tau_{i}\tau_{i+1}\sigma_{i+1}^{-1}\tau_{i}\sigma_{i}\sigma_{i+1}\tau_{i+1}\tau_{i}=\sigma_{i+1}\tau_{i}\sigma_{i+1}\tau_{i+1}\tau_{i}\sigma_{i+1}^{-1}$$
hence
$$\xi_{n,1}\left(a_{n}\left(\sigma_{i}\right)\left(x_{i+1}\right)\right)=\sigma_{i+1}\xi_{n,1}\left(x_{i+1}\right)\sigma_{i+1}^{-1}.$$
\end{proof}

We thus define a Long-Moody construction associated with this parameter $\xi_{n,1}$ and denote it by $\mathbf{LM}_{1}$. The condition \eqref{eq:cond1} is however quite restrictive and the current example follows the spirit of the original Long-Moody construction \cite{Long1}: we therefore focus on the study of $\mathbf{LM}_{1}$ in this paper. By the way we do not claim to be exhaustive here concerning the study of other possible Long-Moody constructions defined by some other compatible parameters $\alpha_{n}$ and $\xi_{n}$.


\subsection{Applications}
In addition to recovering the unreduced Burau representation, the second iteration of the original Long-Moody construction recovers the Lawrence-Krammer representation \cite{bigelowbraid,KrammerLK,lawrencehomological} as a subrepresentation (see \cite[Corollary 2.10]{Long1} or \cite[Section 2.3.1]{soulie1}). The linearity of the welded braid groups being an open problem, there was a hope to construct an analogue of the Lawrence-Krammer representation for these groups. Unfortunately, if we can reconstruct the unreduced Burau representation (see \S\ref{sec:recoverBurau}), the construction $\mathbf{LM}_{1}$ does not produce something new at the second iteration (see \S\ref{sec:iterationBurau}).

On another note, the variants of the Long-Moody construction introduced in \cite{soulie1} allow to recover the Tong-Yang-Ma representation for braid groups (see \cite[Section 2.3.2]{soulie1}). However, we prove in \S\ref{sec:iterationBurau} that it is impossible to reconstruct the Tong-Yang-Ma representation for welded braid groups using \emph{any} Long-Moody construction.

\begin{notation}
Let $R$ be a ring and $r$ an invertible element of $R$. We denote by $r:\mathbf{wB}_{n}\rightarrow R^{\times}$ the one-dimensional representation defined by $r\left(\sigma_{i}\right)$ is the multiplication by $r$ and $r\left(\tau_{i}\right)$ is the identity for all $i\in\left\{ 1,\ldots,n-1\right\} $. Also for $\rho:\mathbf{wB}_{n}\rightarrow GL_{R}\left(V\right)$ a representation, we denote by $r\rho:\mathbf{wB}_{n}\rightarrow GL_{R}\left(V\right)$ the tensor representation $r\underset{R}{\otimes}\rho$.
\end{notation}

\subsubsection{Recovering the Burau representation}\label{sec:recoverBurau}

We start from the one-dimensional representation $t:\mathbf{wB}_{n}\rightarrow \mathbb{Z}\left[t^{\pm1}\right]^{\times}$. We obtain that:
\begin{thm}\label{thm:recbur}
The representation $t^{-1}\mathbf{LM}_{1}\left(t\right)$ is equivalent to $Bur$.
\end{thm}

\begin{proof}
The key point to determine the form of the matrices of $t^{-1}\mathbf{LM}_{1}\left(t\right)$ is to understand the Artin homomorphism on the augmentation ideal. We compute that for all $i\in\left\{ 1,\ldots,n-1\right\}$ and $j\in\left\{ 1,\ldots,n\right\} $, $a_{n}\left(\sigma_{i}\right)$ sends $x_{j}-1$ to
$$\begin{cases}
x_{i+1}-1 & \textrm{if \ensuremath{j=i}}\\
\left(x_{i}-1\right)x_{i+1}+\left(x_{i+1}-1\right)\left(1-x_{i+1}^{-1}x_{i}x_{i+1}\right) & \textrm{if \ensuremath{j=i+1}}\\
x_{j}-1 & \textrm{if \ensuremath{j\notin\left\{ i,i+1\right\}}}
\end{cases}$$
and $a_{n}\left(\tau_{i}\right)$ sends $x_{j}-1$ to $$\begin{cases}
x_{i+1}-1 & \textrm{if \ensuremath{j=i}}\\
x_{i}-1 & \textrm{if \ensuremath{j=i+1}}\\
x_{j}-1 & \textrm{if \ensuremath{j\notin\left\{ i,i+1\right\} \textrm{.}}}
\end{cases}$$
Moreover, for all $k\in\left\{ 1,\ldots,n\right\}$, we denote $\mathbb{Z}\left[t^{\pm1}\right]{}_{k}=\mathbb{Z}\left[\left(x_{k}-1\right)\right]\underset{\mathbb{Z}\left[\mathbf{F}_{n}\right]}{\varotimes}\mathbb{Z}\left[t^{\pm1}\right]$ associated to the generator $x_{k}$ of $\mathbf{F}_{n}$. Since the augmentation ideal $\mathcal{I}_{\mathbf{F}_{n}}$ is a free $\mathbf{F}_{n}$-module on the set $\left\{ x_{k}-1,k\in\left\{ 1,\ldots,n\right\} \right\} $, we thus define an isomorphism
$$\begin{array}{ccc}
\varLambda:\mathcal{I}_{\mathbf{F}_{n}}\underset{\mathbf{F}_{n}}{\varotimes}\mathbb{Z}\left[t^{\pm1}\right] & \longrightarrow & \bigoplus_{k=1}^{n}\mathbb{Z}\left[t^{\pm1}\right]_{k}\\
\left(x_{k}-1\right)\underset{\mathbf{F}_{n}}{\varotimes}v & \longmapsto & \left(0,\ldots,0,\overset{\overset{k\textrm{-}th}{\overbrace{}}}{v},0,\ldots,0\right).
\end{array}$$
Note that $t\left(\xi_{n,1}\left(x_{i}\right)\right)=t$ for all $i\in\left\{ 1,\ldots,n\right\}$. Then result then directly follows from writing the obtained matrix through $\varLambda$.
\end{proof}

\subsubsection{Iteration on the Burau representations}\label{sec:iterationBurau}

We follow here the iteration procedure of \cite{Long1}. In particular, an attempt to define an analogue of the Lawrence-Krammer requires to consider one more variable compared to the Burau representation, i.e. to work on the ring of Laurent polynomials in two variables $\mathbb{Z}\left[t^{\pm1},q^{\pm1}\right]$. For that purpose, we iterate $\mathbf{LM}_{1}$ on the tensor product of the Burau representation with a one-dimensional representation in the new variable.

Also, from now on, we specify by an index in the notation which parameter we consider (i.e. $t$ or $q$) and the dimension of the welded braid group we consider for the Burau representation.

Finally, for convenience of computations reasons, we prefer to use the dual version of Burau representation $Bur^{*}_{n+1,t}$ as input representation and consider $q^{-1}\mathbf{LM}_{1}\left(qBur_{n+1,t}^{*}\right)$. We explain below why this choice does not impact the results presented here. We denote by $\left\{ e_{i}\right\} _{i\in\left\{ 1,\ldots,n+1\right\} }$ the basis for the matrices of the representation $Bur^{*}_{n+1,t}$.

First we prove that this iterate of the Long-Moody construction automatically admits the Burau representation as subrepresentation.

\begin{prop}\label{subrepLM}
The submodule of $\mathcal{I}_{\mathbf{F}_{n}}\underset{\mathbf{F}_{n}}{\varotimes}\mathbb{Z}\left[t^{\pm1},q^{\pm1}\right]^{\oplus n}$ generated by the elements $\left\{ \left(x_{k}-1\right)\otimes e_{1}\right\} _{k\in\left\{ 1,\ldots,n\right\} }$ is closed under the action of $\mathbf{wB}_{n}$. The induced subrepresentation is isomorphic to $Bur_{n,qt}$.
\end{prop}

\begin{proof}
First of all, note that $qBur_{n+1,t}^{*}\left(\sigma_{i+1}\right)\left(e_{1}\right)=qBur_{n+1,t}^{*}\left(\tau_{i+1}\right)\left(e_{1}\right)=e_{1}$ for all $i\in\left\{ 1,\ldots,n\right\}$. Therefore the action of Artin homomorphism gives that $q^{-1}\mathbf{LM}_{1}\left(qBur_{n+1,t}^{*}\right)\left(\sigma_{i}\right)\left(\left(x_{k}-1\right)\otimes e_{1}\right)$ is equal to
$$\begin{cases}
\left(x_{i+1}-1\right)\otimes e_{1} & \textrm{if \ensuremath{j=i}}\\
\left(x_{i}-1\right)\otimes qBur_{n+1,t}^{*}\left(\xi_{n}\left(x_{i+1}\right)\right)e_{1}\\
+\left(x_{i+1}-1\right)\otimes qBur_{n+1,t}^{*}\left(\xi_{n}\left(1-x_{i+1}^{-1}x_{i}x_{i+1}\right)\right)e_{1} & \textrm{if \ensuremath{j=i+1}}\\
\left(x_{j}-1\right)\otimes e_{1} & \textrm{otherwise}
\end{cases}$$
and$$q^{-1}\mathbf{LM}_{1}\left(qBur_{n+1,t}^{*}\right)\left(\tau_{i}\right)\left(\left(x_{k}-1\right)\otimes e_{1}\right)=\begin{cases}
\left(x_{i+1}-1\right)\otimes e_{1} & \textrm{if \ensuremath{j=i}}\\
\left(x_{i}-1\right)\otimes e_{1} & \textrm{if \ensuremath{j=i+1}}\\
\left(x_{j}-1\right)\otimes e_{1} & \textrm{if \ensuremath{j\notin\left\{ i,i+1\right\} \textrm{.}}}
\end{cases}$$
The result thus follows from the fact that $qBur_{n+1,t}^{*}\left(\xi_{n}\left(x_{j}\right)\right)=qt$ for all $j\in\left\{ 1,\ldots,n\right\} $ and the use of the canonical isomorphism $\varLambda$ (see the proof of Theorem \ref{thm:recbur}).
\end{proof}

It remains to identify the quotient of $q^{-1}\mathbf{LM}_{1}\left(qBur_{n+1,t}^{*}\right)$ by the subrepresentation $Bur_{n,qt}$. For that purpose, let us first study the case of $n=3$. The matrix of the morphism $q^{-1}\mathbf{LM}_{1}\left(qBur^{*}_{4,t}\right)\left(\sigma_{1}\right)$ is
$$\left[\begin{array}{cccccccccccc}
0 & 0 & 0 & 0 & qt & q\left(1-t\right) & 0 & 0 & 0 & 0 & 0 & 0\\
0 & 0 & 0 & 0 & 0 & q\left(1-t\right) & qt & 0 & 0 & 0 & 0 & 0\\
0 & 0 & 0 & 0 & 0 & q & 0 & 0 & 0 & 0 & 0 & 0\\
0 & 0 & 0 & 0 & 0 & 0 & 0 & q & 0 & 0 & 0 & 0\\
1 & 0 & 0 & 0 & 1-qt & qt^{-1}\left(1-t-t^{2}+t^{3}\right) & qt\left(1-t\right) & 0 & 0 & 0 & 0 & 0\\
0 & 1-t & t & 0 & 0 & \left(1-q\right)\left(1-t\right) & t\left(1-q\right) & 0 & 0 & 0 & 0 & 0\\
0 & 1 & 0 & 0 & 0 & 1-q & 0 & 0 & 0 & 0 & 0 & 0\\
0 & 0 & 0 & 1 & 0 & 0 & 0 & 1-q & 0 & 0 & 0 & 0\\
0 & 0 & 0 & 0 & 0 & 0 & 0 & 0 & 1 & 0 & 0 & 0\\
0 & 0 & 0 & 0 & 0 & 0 & 0 & 0 & 0 & 1-t & t & 0\\
0 & 0 & 0 & 0 & 0 & 0 & 0 & 0 & 0 & 1 & 0 & 0\\
0 & 0 & 0 & 0 & 0 & 0 & 0 & 0 & 0 & 0 & 0 & 1
\end{array}\right]$$
and the one of $q^{-1}\mathbf{LM}_{1}\left(qBur^{*}_{4,t}\right)\left(\sigma_{2}\right)$ is
$$\left[\begin{array}{cccccccccccc}
1 & 0 & 0 & 0 & 0 & 0 & 0 & 0 & 0 & 0 & 0 & 0\\
0 & 1 & 0 & 0 & 0 & 0 & 0 & 0 & 0 & 0 & 0 & 0\\
0 & 0 & 1-t & t & 0 & 0 & 0 & 0 & 0 & 0 & 0 & 0\\
0 & 0 & 1 & 0 & 0 & 0 & 0 & 0 & 0 & 0 & 0 & 0\\
0 & 0 & 0 & 0 & 0 & 0 & 0 & 0 & qt & 0 & q\left(1-t\right) & 0\\
0 & 0 & 0 & 0 & 0 & 0 & 0 & 0 & 0 & q & 0 & 0\\
0 & 0 & 0 & 0 & 0 & 0 & 0 & 0 & 0 & 0 & q\left(1-t\right) & qt\\
0 & 0 & 0 & 0 & 0 & 0 & 0 & 0 & 0 & 0 & q & 0\\
0 & 0 & 0 & 0 & 1 & 0 & 0 & 0 & 1-qt & 0 & qt^{-1}\left(1-t-t^{2}+t^{3}\right) & qt\left(1-t\right)\\
0 & 0 & 0 & 0 & 0 & 1 & 0 & 0 & 0 & 1-q & 0 & 0\\
0 & 0 & 0 & 0 & 0 & 0 & 1-t & t & 0 & 0 & \left(1-q\right)\left(1-t\right) & t\left(1-q\right)\\
0 & 0 & 0 & 0 & 0 & 0 & 1 & 0 & 0 & 0 & 1-q & 0
\end{array}\right].$$
Those for $\tau_{1}$ and $\tau_{2}$ are the analogues with $t=q=1$. Then the quotient by the subspace $\left\{ e_{1},e_{5},e_{9}\right\} $ gives the matrices
$$\left[\begin{array}{ccccccccc}
0 & 0 & 0 & q\left(1-t\right) & qt & 0 & 0 & 0 & 0\\
0 & 0 & 0 & q & 0 & 0 & 0 & 0 & 0\\
0 & 0 & 0 & 0 & 0 & q & 0 & 0 & 0\\
1-t & t & 0 & \left(1-q\right)\left(1-t\right) & t\left(1-q\right) & 0 & 0 & 0 & 0\\
1 & 0 & 0 & 1-q & 0 & 0 & 0 & 0 & 0\\
0 & 0 & 1 & 0 & 0 & 1-q & 0 & 0 & 0\\
0 & 0 & 0 & 0 & 0 & 0 & 1-t & t & 0\\
0 & 0 & 0 & 0 & 0 & 0 & 1 & 0 & 0\\
0 & 0 & 0 & 0 & 0 & 0 & 0 & 0 & 1
\end{array}\right]$$
$$\left[\begin{array}{ccccccccc}
1 & 0 & 0 & 0 & 0 & 0 & 0 & 0 & 0\\
0 & 1-t & t & 0 & 0 & 0 & 0 & 0 & 0\\
0 & 1 & 0 & 0 & 0 & 0 & 0 & 0 & 0\\
0 & 0 & 0 & 0 & 0 & 0 & q & 0 & 0\\
0 & 0 & 0 & 0 & 0 & 0 & 0 & q\left(1-t\right) & qt\\
0 & 0 & 0 & 0 & 0 & 0 & 0 & q & 0\\
0 & 0 & 0 & 1 & 0 & 0 & 1-q & 0 & 0\\
0 & 0 & 0 & 0 & 1-t & t & 0 & \left(1-q\right)\left(1-t\right) & t\left(1-q\right)\\
0 & 0 & 0 & 0 & 1 & 0 & 0 & 1-q & 0
\end{array}\right].$$
Writing down the matrices of the tensor product of representations
$$Bur_{3,t}^{*}\underset{\mathbb{Z}\left[t^{\pm1},q^{\pm1}\right]}{\otimes}Bur_{3,q}$$
for the generators of the welded braid group, we then recognise the above matrices: we thus identify the quotient of the Long-Moody construction to a tensor product of two Burau representations.

The situation for any $n\geq 3$ works in the same way: indeed consecutive generators $\left(\sigma_{i},\tau_{i}\right)$ and $\left(\sigma_{i+1},\tau_{i+1}\right)$ in $\mathbf{wB}_{n}$ can be identified to $\mathbf{wB}_{3}$. Then the matrices for the quotient of $q^{-1}\mathbf{LM}_{1}\left(qBur^{*}_{n+1,t}\right)$ are similar to the above ones, except that there are more diagonal matrix blocks given by $Bur_{3,t}$. Hence we proved that the iteration of the Long-Moody construction does not construct any new representation of $\mathbf{wB}_{n}$; more precisely we have that:

\begin{thm}\label{thm:iter}
Applying the Long-Moody construction $\mathbf{LM}_{1}$ to the representation $Bur^{*}_{n+1,t}$ gives the short exact sequence of $\mathbf{wB}_{n}$-representations:
\begin{equation}\label{eq:sesiterateburau}
\xymatrix{0\ar@{->}[r] & Bur_{n,qt}\ar@{->}[r] & q^{-1}\mathbf{LM}_{1}\left(qBur^{*}_{n+1,t}\right)\ar@{->}[r] & Bur_{3,t}\underset{\mathbb{Z}\left[t^{\pm1},q^{\pm1}\right]}{\otimes}Bur_{3,q}^{*}\ar@{->}[r] & 0}.
\end{equation}
\end{thm}

\paragraph{Cases of the dual and reduced versions.}
The main conclusion on the iteration of the Long-Moody construction of Theorem \ref{thm:iter} remains true if we apply $\mathbf{LM}_{1}$ to the other versions $Bur_{n+1,t}$ and $\overline{Bur}_{n+1,t}$ of the Burau representation.

For the dual Burau representation $Bur_{n+1,t}$, it is defined at the level of the matrices by the transpose of the inverse of $Bur^{*}_{n+1,t}$. Therefore the $\left(n+1\right)\times\left(n+1\right)$-blocks of $q^{-1}\mathbf{LM}_{1}\left(qBur_{n+1,t}\right)$ are the transpose of those of $q^{-1}\mathbf{LM}_{1}\left(qBur^{*}_{n+1,t}\right)$. Recall that $\left\{ e_{i}\right\} _{i\in\left\{ 1,\ldots,n+1\right\} }$ denotes the basis for the matrices of the representation $Bur^{*}_{n+1,t}$. Then the analogue of Proposition \ref{subrepLM} shows that the submodule of $\mathcal{I}_{\mathbf{F}_{n}}\underset{\mathbf{F}_{n}}{\varotimes}\mathbb{Z}\left[t^{\pm1},q^{\pm1}\right]^{\oplus n}$ generated by the elements $\left\{ \left(x_{k}-1\right)\otimes e_{j}\right\} _{\left(k,j\right)\in\left\{ 1,\ldots,n\right\} \times\left\{ 2,\ldots,n\right\} }$ is closed under the action of $\mathbf{wB}_{n}$ by $q^{-1}\mathbf{LM}_{1}\left(qBur_{n+1,t}\right)$. Moreover repeating mutatis mutandis the proof of Theorem \ref{thm:iter}, we prove that there is a short exact sequence of $\mathbf{wB}_{n}$-representations:
$$\xymatrix{0\ar@{->}[r] & Bur_{3,t}^{*}\underset{\mathbb{Z}\left[t^{\pm1},q^{\pm1}\right]}{\otimes}Bur_{3,q}\ar@{->}[r] & q^{-1}\mathbf{LM}_{1}\left(qBur_{n+1,t}\right)\ar@{->}[r] & Bur_{n,qt}^{*}\ar@{->}[r] & 0}.$$

For the reduced Burau representation $\overline{Bur}_{n+1,t}$, note that we have for all $i\in\left\{ 1,\ldots,n-1\right\} $
$$r_{n+1}\circ Bur^{*}_{n+1,t}\left(\sigma_{i}\right)\circ r_{n+1}^{-1}=\left[\begin{array}{cc}
\overline{Bur}^{*}_{n+1,t}\left(\sigma_{i}\right) & 0\\
L_{i} & 1
\end{array}\right]$$
where $L_{i}=\left[\begin{array}{ccccc}0 & \cdots & 0 & \delta_{i,n} & 1\end{array}\right]$ and $r_{n+1}$ is the $n\times n$-matrix with coefficients $r_{i,j}=1$ if $j\leq i$ and $r_{i,j}=0$. Therefore, along the injection $id_{1}*-:\mathbf{wB}_{n}\hookrightarrow\mathbf{wB}_{n+1}$, the short exact sequence \eqref{eq:dualsesburau} splits as a short exact sequence of $\mathbf{wB}_{n}$-representations. Then it follows from the freeness (and a fortiori flatness) of the augmentation ideal $\mathcal{I}_{\mathbf{F}_{n}}$ as a $\mathbb{Z}\left[\mathbf{F}_{n}\right]$-module that we have an isomorphism of $\mathbf{wB}_{n}$-representations:
$$\ensuremath{q^{-1}\mathbf{LM}_{1}\left(qBur_{n+1,t}^{*}\right)}\cong q^{-1}\mathbf{LM}_{1}\left(q\overline{Bur}_{n+1,t}^{*}\right)\oplus q^{-1}\mathbf{LM}_{1}\left(q\mathbb{Z}\left[t^{\pm1}\right]\right).$$
The representation $q^{-1}\mathbf{LM}_{1}\left(q\overline{Bur}_{n+1,t}^{*}\right)$ is thus determined by the short exact sequence \eqref{eq:sesiterateburau}.

\subsubsection{On the impossibility to recover the Tong-Yang-Ma representation}\label{sec:notrecoverTYM}

We detail in \S\ref{sec:theoretical} that a Long-Moody construction is equivalent to the setting of two parameters $a_{n}:\mathbf{wB}_{n}\rightarrow Aut\left(\mathbf{F}_{n}\right)$ and $\xi_{n}:\mathbf{F}_{n}\rightarrow\mathbf{wB}_{n+1}$ satisfying the compatibility condition \eqref{eq:diagram}. So far, we have used the Artin homomorphism for the action $a_{n}$ and this parameter determines the shape of the obtained matrices (see \S\ref{sec:recoverBurau} and \S\ref{sec:iterationBurau}).

In \cite[Section 2.3.2]{soulie1}, the Tong-Yang-Ma representation for braid groups is recovered by playing on the choice of this morphism. However, we prove that it is not the case for welded braid groups in the following result. In particular, we call the representation which space is $\mathbb{Z}\left[t^{\pm1}\right]^{\oplus n}$ and on which $\mathbf{wB}_{n}$ acts trivially the \emph{trivial $n$-dimensional representation} of $\mathbf{wB}_{n}$. 

\begin{thm}\label{impossiblerecoverTYM}
Let $\mathbf{LM}$ be the Long-Moody construction associated with a Wada representation for a welded braid groups and an abstract compatible $\xi_{n}$. Then $t^{-1}\mathbf{LM}\left(t\right)$ is equivalent either to the Burau representation $Bur$ (or its dual $Bur^{*}$), or to the permutation representation extended to $\mathbf{wB}_{n}$, or to the trivial $n$-dimensional representation of $\mathbf{wB}_{n}$.

In particular, the Tong-Yang-Ma representation (or its dual) cannot be recovered by any Long-Moody construction for welded braid groups.
\end{thm}

\begin{proof}
We recall from \S\ref{sec:Wada} that we can only consider four Wada representations (Types 2, 3, 4 and 5). First of all, restricting to $\mathbf{B}_{n}$, \cite[Section 2.3.2]{soulie1} automatically implies that:
\begin{itemize}
\item with the Type $3$ Wada representation, $t^{-1}\mathbf{LM}\left(t\right)$ is equivalent to the permutation representation extended to $\mathbf{wB}_{n}$;
\item with the Type $4$ or $5$ Wada representation, $t^{-1}\mathbf{LM}\left(t\right)$ is equivalent the Burau representation.
\end{itemize}
\cite[Section 2.3.2]{soulie1} uses the Type $2$ Wada representation to obtain the representation $TYM$ for braid groups. Nevertheless, with the extension to $\mathbf{wB}_{n}$, it follows from the compatibility condition \eqref{eq:cond1} that $\tau_{2}\xi_{n}\left(x_{1}\right)\tau_{2}=\xi_{n}\left(x_{2}\right)=\sigma_{2}^{-1}\xi_{n}\left(x_{1}\right)\sigma_{2}$ and that $\xi_{n}\left(x_{2}^{-1}\right)=\sigma_{2}\xi_{n}\left(x_{1}\right)\sigma_{2}^{-1}$. We deduce that $\xi_{n}\left(x_{2}^{-1}\right)=\sigma_{2}^{-2}\xi_{n}\left(x_{1}\right)\sigma_{2}^{2}$ and that $\xi_{n}\left(x_{1}\right)=\left(\tau_{2}\sigma_{2}\right)^{-1}\xi_{n}\left(x_{1}\right)\tau_{2}\sigma_{2}$. Hence we obtain that:$$\xi_{n}\left(x_{1}\right)=\sigma_{2}^{-1}\tau_{2}^{2}\sigma_{2}^{-1}\xi_{n}\left(x_{1}\right)\sigma_{2}\tau_{2}^{2}\sigma_{2}=\xi_{n}\left(x_{1}^{-1}\right).$$
Then $\left(t\circ\xi_{n}\right)^{2}=1$ since $t$ and $\xi_{n}$ are morphisms. The only order $2$ elements of $\mathbb{Z}\left[t^{\pm1}\right]$ are $1$ and $-1$ and it follows from the definition of $t$ that $-1$ is not in its image. A fortiori the composite $t\circ\xi_{n}$ is the trivial morphism $\mathbf{F}_{n}\rightarrow\mathbf{wB}_{n+1}$. Then it follows from a straightforward computation of the matrices that $t^{-1}\mathbf{LM}\left(t\right)$ is equivalent to the permutation representation extended to $\mathbf{wB}_{n}$.

Note that a Long-Moody construction multiplies by $n$ the dimension of an input representation. Therefore the only way to construct the Tong-Yang-Ma representation would be to start from a one-dimensional representation. Also, the form of the matrix 
$$Id_{i-1}\oplus\left[\begin{array}{cc}
0 & 1\\
t & 0
\end{array}\right]\oplus Id_{n-i-1}
$$
for each braid generator $\sigma_{i}$ implies that the parameter $\alpha_{n}$ of the Long-Moody construction which would recover the Tong-Yang-Ma representation has to be a Wada representation.
Therefore the above study of $t^{-1}\mathbf{LM}\left(t\right)$ proves that no Long-Moody construction can recover the Tong-Yang-Ma representation.
\end{proof}

Therefore, Artin and, more generally, Wada representations do not seem to be a useful tool to obtain interesting linear representations for welded braid groups,
at least in the framework of Long-Moody procedure; a possible lead should be to consider other free groups embedded in $\mathbf{wB}_{n}$
or other actions. For instance, recently in \cite{DFM} was constructed a lift of Artin representation to an action at the $\pi_2$-level (remind that $B^3$ less $n$ trivial circles is not aspherical) and it seems interesting to adapt Long-Moody procedure in this framework. It seems also clear that a deeper understanding of $\mathbf{wB}_{n}$ and of its subgroup $\mathbf{D}_{n-1}$ in terms of \emph{motion groups} of circles \cite{damiani2017journey} could provide new perspectives. On another hand, an homological approaches to construct linear representations of welded braid groups is set in \cite{souliepalmer}: the connection with the representations of the present paper would deserve to be explored.


\bibliographystyle{plain}

\bibliography{bibliographie}


\end{document}